\documentclass{article}
\usepackage{amsmath}
\usepackage{amsthm}
\usepackage{amsfonts}
\usepackage{enumerate}
\usepackage[utf8]{inputenc}
\usepackage[T1]{fontenc}
\usepackage{amssymb}
\usepackage{authblk}

\newcommand{\R}{\mathbb{R}}
\newcommand{\SA}{\mathcal{S}}
\newcommand{\GA}{\mathcal{G}}
\newcommand{\linear}{\mathcal{L}}
\newcommand{\On}{\mathcal{O}}
\newcommand{\SO}{\mathrm{SO}}
\newcommand{\dg}{\mathrm{diag}}

\newcommand{\Dg}{\mathrm{d}}
\newcommand{\Rn}{\mathbb{R}^{n\times n}}
\newcommand{\RN}{\mathbb{R}^{N\times N}}

\newcommand{\deter}{\mathrm{det}}

\newcommand{\uni}{\mathcal{U}}
\newcommand{\tr}{\mathrm{tr}}

\title{Convexity and Star-shapedness of Real Linear Images of Special Orthogonal Orbits}
\author[1]{Pan-Shun Lau\thanks{panlau@hku.hk}}
\author[1]{Tuen-Wai Ng\thanks{ntw@maths.hku.hk}}
\author[1]{Nam-Kiu Tsing\thanks{nktsing@hku.hk}}
\affil[1]{\small Department of Mathematics, The University of Hong Kong, Pokfulam, Hong Kong}

\begin{document}
\theoremstyle{plain}
\newtheorem{theorem}{Theorem}[section]
\newtheorem{corollary}[theorem]{Corollary}
\newtheorem{lemma}[theorem]{Lemma}
\newtheorem{proposition}[theorem]{Proposition}
\theoremstyle{definition}
\newtheorem{definition}{Definition}
\newtheorem{example}{Example}
\newtheorem{problem}{Problem}
\theoremstyle{remark}
\newtheorem{remark}{Remark}

\maketitle
\hrule
\section*{\small Abstract}

Let $A\in\Rn$ and $\SO_n:=\{U\in\Rn:UU^t=I_n,\deter U>0\}$ be the set of $n\times n$ special orthogonal matrices. Define the (real) special orthogonal orbit of $A$ by
\[
O(A):=\{UAV:U,V\in\SO_n\}.
\] In this paper, we show that the linear image of $O(A)$ is star-shaped with respect to the origin for arbitrary linear maps $L:\Rn\to\R^\ell$ if $n\geq 2^{\ell-1}$. In particular, for linear maps $L:\Rn\to\R^2$ and when $A$ has distinct singular values, we study $B\in O(A)$ such that $L(B)$ is a boundary point of $L(O(A))$. This gives an alternative proof of a result by Li and Tam on the convexity of $L(O(A))$ for linear maps $L:\Rn\to\R^2$.\\[5pt]
{\footnotesize\emph{AMS Classification:}	15A04, 15A18.\\
\emph{Keywords:} linear transformation, special orthogonal orbits, convexity, star-shapedness}
\vspace{10pt}
\hrule

\section{Introduction}

Let $\On_n:=\{U\in\Rn:U^tU=UU^t=I_n\}$ and $\SO_n:=\{U\in \On_n: \deter U>0\}$ be the sets of $n\times n$ orthogonal matrices and $n\times n$ special orthogonal matrices respectively. For any $A\in\Rn$, we define the special orthogonal orbit of $A$ by
\[
O(A):=\{UAV:U,V\in \SO_n\}.
\]
It is clear that every element in $O(A)$ has the same collection of singular values and the same sign of determinant. In \cite{Thompson77}, Thompson studied the set of diagonals of the matrices in $O(A)$, and in \cite{Miranda}, Miranda and Thompson studied the characterizations of extreme values of $L(O(A))$ where $L:\R^{n\times n} \to \R$ is a linear map.

A set $S$ is said to be star-shaped with respect to $c\in S$ if for all $0\leq \alpha\leq 1$ and $x\in S$, $\alpha x+(1-\alpha)c\in S$. The $c$ is called a star center of $S$. In this paper, we shall study the star-shapedness of images of $O(A)$ under arbitrary linear maps $L:\Rn\to\R^\ell$. 

In fact the study of linear images of matrix orbits is a popular topic. If $A,C$ are $n\times n$ complex matrices and $\uni_n$ denotes the group of $n\times n$ (complex) unitary matrices, then the (classical) numerical range of $A$, denoted by $W(A)$, and the $C$-numerical range of $A$, denoted by $W_C(A)$, are simply the images of the unitary orbit of $A$, denoted by
\[\uni_n(A):=\{U^*AU:U\in\uni_n\},
\]under the linear maps
\[
X\longmapsto \tr(E_1 X) \;\;\;\mathrm{and}\;\;\; X\longmapsto \tr(CX)
\]respectively, where $E_1$ is the diagonal matrix with diagonal entries $1,0,...,0.$ It has been proved that $W(A)$ is always convex and $W_C(A)$ is always star-shaped (see \cite{CheungT}, \cite{Hausdorff}, \cite{Toeplitz}). Many results on the convexity and the star-shapedness of other generalized numerical ranges, which can be expressed as some particular linear images of $\uni_n(A)$, have been obtained (e.g., see \cite{CheungT}, \cite{HJ}, \cite{ckli}, \cite{LP1999}, \cite{LP2011}, \cite{Tsing81}, \cite{Westwick}).

Our paper is organized as follows. In Section 2, we study an inclusion relation of $L(O(A))$ with $L:\Rn\to\R^\ell$ and $n\geq 2^{\ell-1}$. We then apply the inclusion relation to show that $L(O(A))$ is star-shaped for general $A$ and $L:\Rn\to\R^\ell$ where $n\geq 2^{\ell-1}$. In particular, the star-shapedness holds for $L(O(A))$ with $L:\Rn\to\R^2$ and $n\geq 3$. Moreover, we shall extend our results to linear images of the following joint (real) orthogonal orbits,
\[
\begin{split}
&\boldsymbol{O}_1(A_1,...,A_m;G):=\{(A_1V,...,A_mV):V\in G\},\\
&\boldsymbol{O}_2(A_1,...,A_m;G):=\{(UA_1,...,UA_m):U\in G\},\\
&\boldsymbol{O}_3(A_1,...,A_m;G):=\{(UA_1V,...,UA_mV):U,V\in G \},
\end{split}
\]
where $G=\On_n$ or $\SO_n$. In Section 3, we study boundary points of $L(O(A))$ with $L:\Rn\to\R^2$. When $A\in\Rn$ has distinct singular values, we shall discuss the conditions on $U,V\in\SO_n$ under which $L(UAV)$ will be a boundary point of $L(O(A))$. Then we show that the intersection of $L(O(A))$ and any of its supporting lines is path-connected. Combining the result in Section 2, convexity of $L(O(A))$ for $L:\Rn\to\R^2$ then follows. This result was proved by Li and Tam \cite{LT} with a different approach. We shall also discuss the convexity of linear images of joint orthogonal orbits.

\section{Star-shapedness of linear image of $O(A)$}



The following is the first main theorem in this section.

\begin{theorem}\label{main_thm2}
Let $\ell\geq 3$. For any $A\in\Rn$ and any linear map $L:\Rn\rightarrow \R^\ell$ with $n\geq 2^{\ell-1}$, $L(O(A))$ is star-shaped with respect to the origin.
\end{theorem}

We need some lemmas to prove Theorem~\ref{main_thm2}. Note that any linear map $L:\Rn \rightarrow \R^\ell$ can be expressed as
\[
L(X)=\big(\tr (P_1X),...,\tr(P_\ell X) \big)^t
\]
for some $P_1,...,P_\ell\in\Rn$. For convenience, for $M\subseteq\Rn$ and any $P_1,...,P_\ell\in\Rn$, we define
\[
\linear(P_1,...,P_\ell ; M):=\big\{\big(\tr (P_1X),...,\tr(P_\ell X) \big)^t:X\in M\big\}.
\]
For $A,P_1,...,P_\ell\in \Rn$, we let $\SA_A(P_1,...,P_\ell)$ be the set containing $(P'_1,...,P'_\ell)$ where $P'_1,...,P'_\ell \in \Rn$ and $\linear(P'_1,...,P'_\ell;O(A))\subseteq \linear(P_1,...,P_\ell;O(A))$. This definition is motivated by Cheung and Tsing \cite{CheungT}. Below are some basic properties of $\SA_A(P_1,...,P_\ell)$. 

\begin{lemma} \label{lemma3}
Let $A\in\Rn$. For any $P_1,...,P_\ell\in \Rn$, the followings hold:
\begin{enumerate}
\item[(a)] $\SA_{XAY}(UP_1V,...,UP_\ell V)=\SA_A(P_1,...,P_\ell)$ for any $U,V,X,Y\in \SO_n$;
\item[(b)] $(UP_1V,...,UP_\ell V)\in \SA_A(P_1,...,P_\ell)$, for any $U,V\in \SO_n$;
\item[(c)] $\SA_A(P'_1,...,P'_\ell)\subseteq \SA_A(P_1,...,P_\ell)$ for any $(P'_1,...,P'_\ell)\in \SA_A(P_1,...,P_\ell)$;
\item[(d)] $\linear(P_1,...,P_\ell;O(A)){=}\big\{\big(\tr (P'_1 A),...,\tr (P'_\ell A) \big)^t{:}(P'_1,...,P'_\ell){\in} ~\SA_A(P_1,...,P_\ell)\big\}$.
\end{enumerate}
\end{lemma}
\begin{proof} (a), (b) and (c) are trivial. For (d), ``$\subseteq$'' follows from (b) and ``$\supseteq$'' follows from the definition of $\SA_A(P_1,...,P_\ell)$.
\end{proof}

\begin{lemma} \label{main_thm1_red} The following statements are equivalent (hence if one of these statements holds then the other three must also hold):
\begin{enumerate}
\item[(a)] $L(O(A))$ is star-shaped with respect to the origin for any $A\in\Rn$ and any linear map $L:\Rn\rightarrow \R^\ell$;
\item[(b)] $\SA_A(P_1,...,P_\ell)$ is star-shaped with respect to $(0_n,...,0_n)$ for any $A\in\Rn$ and any $P_1,...,P_\ell\in\Rn$, where $0_n$ is the $n\times n$ zero matrix;
\item[(c)] $L(\SO_n)$ is star-shaped with respect to the origin for any linear map $L:\Rn\rightarrow \R^\ell$;
\item[(d)] $\SA_{I_n}(P_1,...,P_\ell)$ is star-shaped with respect to $(0_n,...,0_n)$ for any $P_1,...,P_\ell\in\Rn$.
\end{enumerate}
\end{lemma}

\begin{proof}
((a)$\Rightarrow$(b)) For any $(P'_1,...,P'_\ell)\in \SA_A(P_1,...,P_\ell)$, $U,V\in\SO_n$ and $0\leq \alpha\leq 1$, we have
\[
\left(\tr(\alpha P'_1UAV),...,\tr(\alpha P'_\ell UAV)\right)^t \in \linear(P'_1,...,P'_\ell;O(A))\subseteq \linear(P_1,...,P_\ell;O(A)).
\]
Hence $\alpha (P'_1,...,P'_\ell)\in \SA_A(P_1,...,P_\ell)$.

((b)$\Rightarrow$(a)) Apply Lemma~\ref{lemma3} (b).

((a)$\Rightarrow$(c)) If we take $A=I_n$, then $O(A)=\SO_n$.

((c)$\Rightarrow$(a)) Let $L:\Rn\rightarrow \R^\ell$ be linear and $A\in\Rn$. For any $U\in\SO_n$, define linear map $L_{UA}:\Rn\rightarrow \R^\ell$ by
\[
L_{UA}(X)=L(UAX).
\]
For any $U,V\in\SO_n$ and $0\leq\alpha\leq 1$, since $L_{UA}(\SO_n)$ is star-shaped with respect to the origin, there exists $V'\in\SO_n$ such that
\[
\alpha L(UAV)= \alpha L_{UA}(V) = L_{UA}(V')=L(UAV')\in L(O(A)).
\]

((c)$\Leftrightarrow$(d)) Apply similar arguments as those in (a)$\Leftrightarrow$(b).
\end{proof}

To prove Theorem~\ref{main_thm2}, we apply Lemma~\ref{main_thm1_red} and show the star-shapedness of $\SA_{I_n}(P_1,...,P_\ell)$ for any $P_1,...,P_\ell\in\Rn$ with $n\geq 2^{\ell-1}$. For simplicity, we denote $\SA_{I_n}(P_1,...,P_\ell)$ by $\SA(P_1,...,P_\ell)$. In fact, by the following lemma, we may focus only on the case of $n=2^{\ell-1}$.


\begin{lemma} \label{lemma5}
If $\SA(\hat{P}_1,...,\hat{P}_\ell)$ is star-shaped with respect to the origin for all $\hat{P}_1,...,\hat{P}_\ell\in\Rn$, then for all $m> n$ and for all $P_1,...,P_\ell\in\R^{m\times m}$, $\SA(P_1,...,P_\ell)$ is star-shaped with respect to the origin.
\end{lemma}
\begin{proof}
Let $m=n+k$ where $k$ is a positive integer. For any $(P'_1,...,P'_\ell)\in \SA(P_1,...,P_\ell)$, we write
\[P'_i=\begin{bmatrix}P'_{i1} & P'_{i2}\\ P'_{i3} & P'_{i4}\end{bmatrix},\] where $P'_{i1}\in\R^{n\times n}$ and $P'_{i4}\in\R^{k\times k}$.
We shall show that $\big(P'_1(\epsilon),...,P'_\ell(\epsilon)\big)\in \SA(P_1,...,P_\ell)$ where $P'_i(\epsilon)= (\epsilon I_n \oplus I_k) P'_i$ and $0\leq \epsilon\leq 1$. For any $U\in\SO_m$, we write 
\[U=\begin{bmatrix}U_1 & U_2\\ U_3 & U_4\end{bmatrix},\]
where $U_1\in\Rn$ and $U_4\in\R^{k\times k}$. Then for $0\leq \epsilon\leq 1$, by the hypothesis of the lemma, there exists $V\in\SO_n$ such that
\[\begin{split}
&~\big(\tr (P'_1(\epsilon)U),...,\tr (P'_\ell(\epsilon) U)\big)^t    \\
=&~ \epsilon\left(\tr (P'_{11}U_1+P'_{12}U_3),..., \tr (P'_{\ell 1}U_1+P'_{\ell 2}U_3)\right)^t\\
																				&\hspace{65pt}  +\left(\tr (P'_{13}U_2+P'_{14}U_4),...,\tr (P'_{\ell 3}U_2+P'_{\ell 4}U_4)\right)^t \\
=&~\Big(\tr \big[(P'_{11}U_1+P'_{12}U_3)V\big],...,\tr \big[(P'_{\ell 1}U_1+P'_{\ell 2}U_3)V\big]\Big)^t\\
																				&\hspace{65pt} +\Big(\tr (P'_{13}U_2+P'_{14}U_4),...,\tr (P'_{\ell 3}U_2+P'_{\ell 4}U_4)\Big)^t \\
=&~\Big(\tr \big[P'_1U(V\oplus I_k)\big],...,\tr \big[P'_\ell U(V\oplus I_k)\big]\Big)^t\\
\in &~ \linear(P'_1,...,P'_\ell;\SO_m)\\
\subseteq &~ \linear(P_1,...,P_\ell; \SO_m).
\end{split}\]
Since this holds for all $U\in\SO_m$, we have $\big(P'_1(\epsilon),...,P'_\ell(\epsilon)\big)\in\SA(P_1,...,P_\ell)$. Note that the preceding result also holds if we multiply arbitrary $n$ rows of $P'_i$ by $0\leq \epsilon \leq 1$. We re-apply the result by considering all $n$-combinations of rows to obtain $\epsilon^N(P'_1,...,P'_\ell)\in \SA(P_1,...,P_\ell)$, where $N=\dfrac{m!}{n!k!}$. For any $0\leq \alpha \leq 1$, we put $\epsilon=\sqrt[N]{\alpha}$ to obtain $\alpha (P'_1,...,P'_\ell)\in \SA(P_1,...,P_\ell)$.\end{proof}
We now consider the following recursively defined matrices. Let
\[
R(\theta_1)=\begin{bmatrix}\cos \theta_1 & \sin\theta_1\\ -\sin\theta_1 &\cos\theta_1\end{bmatrix}
\]
and
\[
R(\theta_1,...,\theta_{k})= \begin{bmatrix}\cos \theta_{k} I_N& \sin\theta_{k}R(\theta_1,...,\theta_{k-1}) \\ -\sin\theta_{k} R(\theta_1,...,\theta_{k-1})^t &\cos\theta_{k}I_N\end{bmatrix}
\] 
where $N=2^{k-1}$. Note that $R(\theta_1,...,\theta_k)\in\SO_{2^{k}}$. 

\begin{lemma} \label{ellipsoid}
Let $\ell\geq 2$ and $P_1,...,P_\ell\in \R^{N\times N}$ where $N=2^{\ell-1}$. Then for any $U,V\in\SO_N$, the set
\[
\begin{split}
& E(U,V):=\\
&\Big\{\Big(\tr\big(R(\theta_1,...,\theta_{\ell-1})UP_1V\big),...,\tr\big(R(\theta_1,...,\theta_{\ell-1})UP_\ell V\big)\Big)^t{:}
	\theta_1,...,\theta_{\ell-1}\in [0,2\pi]
\Big\}\end{split}
\]
is an ellipsoid in $\R^\ell$ centered at the origin and is a subset of $\linear(P_1,...,P_\ell;\SO_N)$.
\end{lemma}
\begin{proof}
We first show that for any $A\in\R^{N\times N}$ where $N=2^{\ell-1}$,
\[
\tr\big(R(\theta_1,...,\theta_{\ell-1})A\big)=\begin{bmatrix}a_1&\cdots& a_\ell\end{bmatrix}\begin{bmatrix}\cos\theta_{\ell-1} \\ \sin\theta_{\ell-1}\cos\theta_{\ell-2}\\ \sin\theta_{\ell-1}\sin\theta_{\ell-2}\cos\theta_{\ell-3} \\ \vdots \\ \sin\theta_{\ell-1}\sin\theta_{\ell-2}\cdots\sin\theta_1\end{bmatrix}
\]
for some $a_1,...,a_\ell\in\R$ by induction on $\ell$. The case for $\ell=2$ is trivial. Now assume it is true for $\ell\leq m$ where $m\geq 2$ and consider $A\in\R^{2M\times 2M}$ where $M=2^{m-1}$. We write 
\[A=\begin{bmatrix}A_1 & A_2\\A_3&A_4\end{bmatrix}\]
where $A_i\in\R^{M\times M},i=1,...,4.$ Then
\[\tr\big(R(\theta_1,...,\theta_{m})A\big)=\cos\theta_{m}\tr(A_1+A_4)+\sin\theta_{m} \tr\big(R(\theta_1,...,\theta_{m-1})(A_3-A_2^t)\big).\] 
By induction assumption on $\tr\big(R(\theta_1,...,\theta_{m-1})(A_3-A_2^t)\big)$, $\tr\big(R(\theta_1,...,\theta_{m})A\big)$ is in the desired form. Hence we have
\[E(U,V)=\left\{T\begin{bmatrix}\cos\theta_{\ell-1} \\ \sin\theta_{\ell-1}\cos\theta_{\ell-2}\\ \sin\theta_{\ell-1}\sin\theta_{\ell-2}\cos\theta_{\ell-3} \\ \vdots \\ \sin\theta_{\ell-1}\sin\theta_{\ell-2}\cdots\sin\theta_1\end{bmatrix}:\theta_1,...,\theta_{\ell-1}\in [0,2\pi]\right\},\]
for some $T\in\R^{\ell \times \ell}$ and hence $E(U,V)$ is an ellipsoid in $\R^\ell$ centered at the origin. As $R(\theta_1,...,\theta_k)$ is a special orthogonal matrix, $E(U,V)\subseteq \linear(P_1,...,P_\ell;\SO_N)$. 
\end{proof}


\begin{lemma}\label{exist_UV}
Let $\ell\geq 3$. For any $P_1,...,P_\ell\in\R^{N\times N}$ where $N=2^{\ell-1}$, there exist $U,V\in\SO_N$ such that $E(U,V)$ defined in Lemma~\ref{ellipsoid} degenerates (i.e., $E(U,V)$ is contained in an affine hyperplane in $\R^\ell$).
\end{lemma}
\begin{proof}
From the proof of Lemma~\ref{ellipsoid}, we see that if there exist $U,V\in\SO_N$ such that 
\[
UP_1V=\begin{bmatrix}P^{(1)}_1 &P^{(1)}_2\\P^{(1)}_3 & P^{(1)}_4\end{bmatrix}
\]where $P^{(1)}_i\in\R^{\frac{N}{2}\times \frac{N}{2}}$, $i=1,...,4$, $\tr(P^{(1)}_1+P^{(1)}_4)=0$ and $P^{(1)}_2=P^{(1)}_3=0$, then the first coordinate of $E(U,V)$ is always $0$ and hence $E(U,V)$ degenerates. Let $U',V'\in\SO_N$ be such that $U'P_1V'=\dg(p_1,...,p_N)$. Then 
\[
U=U',\;\;\; V=V'\left(\begin{bmatrix}0&-1\\1&0\end{bmatrix}\oplus\cdots\oplus\begin{bmatrix}0&-1\\1&0\end{bmatrix}\right)
\] will give the desired $UP_1V$.
\end{proof}

Note that, by considering $P_1=\begin{bmatrix} 1 &0\\ 0 &0\end{bmatrix}$ and $P_2=\begin{bmatrix} 0 &0\\ 1 &0\end{bmatrix}$, then for any $U,V\in\SO_2$, the ellipse $E(U,V)$ defined in Lemma~\ref{ellipsoid} is always non-degenerate. Hence Lemma~\ref{exist_UV} and Theorem~\ref{main_thm2} fail to hold for $\ell=2$. 

We are now ready to prove our first main result.

\begin{proof}[Proof of Theorem~\ref{main_thm2}]
By Lemma~\ref{main_thm1_red} and Lemma~\ref{lemma5}, it suffices to show that for any $P_1,...,P_\ell\in\RN$ with $N=2^{\ell-1}$, $\SA(P_1,...,P_\ell)$ is star-shaped with respect to $(0_N,...,0_N)$. Let $(P'_1,...,P'_\ell)\in\SA(P_1,...,P_\ell)$ and $0\leq \alpha\leq 1$. For any $U\in\SO_N$, we define $E(I_N,U)$ as in Lemma~\ref{ellipsoid}. If $\alpha\big(\tr(P'_1U),...,\tr(P'_1U)\big)^t \in E(I_N,U)$, then we have
\[
\alpha\big(\tr(P'_1U),...,\tr(P'_1U)\big)^t\in\linear(P'_1,...,P'_\ell;\SO_N)\subseteq\linear(P_1,...,P_\ell;\SO_N).
\]
Assume now $\alpha\big(\tr(P'_1U),...,\tr(P'_1U)\big)^t\notin E(I_N,U)$. As the center of $E(I_N,U)$ is the origin, we have $\alpha\big(\tr(P'_1U),...,\tr(P'_1U)\big)^t$ lies inside the ellipsoid $E(I_N,U)$. As $\SO_N\times \SO_N$ is path connected, consider a continuous function $f:[0,1]\to \SO_N\times \SO_N$ with $f(0)=(I_N,U)$ and $f(1)=(U',V')$ where $(U',V')$ are defined in Lemma~\ref{exist_UV}. Then by continuity of $f$, there exists $s\in[0,1]$ such that $\alpha\big(\tr(P'_1U),...,\tr(P'_1U)\big)^t\in E(f(s))\subseteq \linear(P'_1,...,P'_\ell;\SO_N)\subseteq \linear(P_1,...,P_\ell;\SO_N)$. As it is true for all $U\in\SO_N$, we have 
\[
\alpha(P'_1,...,P'_\ell)+(1-\alpha)(0_n,...,0_n)=\alpha(P'_1,...,P'_\ell)\in \SA(P_1,...,P_\ell).
\]
\end{proof}

In fact for $\ell=2$, we have the following theorem, the proof of which is given by Lemma~\ref{ellipse1} to Corollary~\ref{include2}.

\begin{theorem} \label{main_thm1}
Let $A\in\Rn$ and $L:\Rn\rightarrow \R^2$ be a linear map with $n\geq 3$. Then $L(O(A))$ is star-shaped with respect to the origin.
\end{theorem}

\begin{lemma}\label{ellipse1}
Let $n\geq 2$. For any $P,Q\in\mathbb{R}^{n\times n}$, $U\in \SO_n$, the locus of the point $\big(\tr(T_\theta PU),\tr(T_\theta QU)\big)^t$ where $T_\theta=R(\theta)\oplus I_{n-2} $ forms an ellipse $E(U)$ in $\mathbb{R}^2$ when $\theta$ runs through $[0,2\pi]$.
\end{lemma}

\begin{proof}
We write
\[
P=\left[\begin{array}{c}
p_{(1)}\\ \hline	
p_{(2)}\\ \hline
P_{(3)}
\end{array}\right],\quad Q=\left[\begin{array}{c}
q_{(1)}\\ \hline	
q_{(2)}\\ \hline
Q_{(3)}
\end{array}\right] \quad \text{     and     } \quad
U=\left[\begin{array}{c|c|c}u^{(1)} & u^{(2)}& U^{(3)} \end{array} \right]
\]
where $p^t_{(1)},p^t_{(2)},q^t_{(1)},q^t_{(2)},u^{(1)},u^{(2)}\in\mathbb{R}^{n}$ and $P_{(3)}^t,Q^t_{(3)},U^{(3)}\in\R^{n\times (n-2)}$.  Direct computation shows
\[
\begin{split}
\tr(T_\theta PU)=\cos\theta( p_{(1)} u^{(1)} + p_{(2)} u^{(2)})  +  \sin\theta (p_{(2)}u^{(1)} - p_{(1)}u^{(2)}) + \text{tr}(P^t_{(3)}U^{(3)}).
\end{split}
\]
Similarly for $\tr(T_\theta QU)$. Hence
\[
\begin{bmatrix} \tr(T_\theta PU)\\ \tr(T_\theta QU)\end{bmatrix} = \begin{bmatrix} p_{(1)} u^{(1)} + p_{(2)} u^{(2)} &p_{(2)}u^{(1)} - p_{(1)}u^{(2)}\\ q_{(1)} u^{(1)} + q_{(2)} u^{(2)}&q_{(2)}u^{(1)} - q_{(1)}u^{(2)} \end{bmatrix}\begin{bmatrix} \cos\theta \\ \sin\theta \end{bmatrix}+ \begin{bmatrix}\tr (P_{(3)}U^{(3)}) \\	\tr (Q_{(3)}U^{(3)}) \end{bmatrix},
\]
the locus of which forms an ellipse (possibly degenerate) when $\theta$ runs through $[0,2\pi]$.
\end{proof}

\begin{lemma}\label{lemma6}
For any $P,Q\in \mathbb{R}^{n\times n}$ with $n\geq 3$, there exists $U_0\in \SO_n$ such that the ellipse $E(U_0)$ defined in Lemma~\ref{ellipse1} degenerates. 
\end{lemma}
\begin{proof}
Note that $E(U)$ degenerates if we find orthonormal vectors $u^{(1)},u^{(2)}\in \R^n$ such that the matrix
\[
\begin{bmatrix} p_{(1)} u^{(1)} + p_{(2)} u^{(2)} &p_{(2)}u^{(1)} - p_{(1)}u^{(2)}\\ q_{(1)} u^{(1)} + q_{(2)} u^{(2)}&q_{(2)}u^{(1)} - q_{(1)}u^{(2)} \end{bmatrix}
\] 
is singular. We will show that for any given $p_1,p_2\in \R^n$, there exist orthonormal vectors $u_1,u_2$ such that $p_1^tu_2 = p_2^tu_1 = p_1^tu_1+p_2^tu_2 = 0$. By scaling and rotating, we assume without loss of generality that $p_1=(1,0,...,0)^t$ and $p_2=(a,b,0,...,0)^t$ where $a,b\in\mathbb{R}$ and $0\leq b\leq 1$. If $a=0$ or $b= 0$, we can take $u_1=(-b,0,\sqrt{1-b^2},0,...,0)^t$ and $u_2=(0,1,0,...,0)^t$. Now, assume that $a\neq 0$ and $0<b\leq 1$. For $\theta\in[0,\pi]$ consider unit vectors
\[
v_\theta=\begin{bmatrix}
0\\
\cos\theta\\
\sin\theta\\
0\\
\vdots\\
0
\end{bmatrix} \;\;\text{  and  } \;\; w_\theta=\frac{1}{\sqrt{b^2\sin^2\theta+a^2}}\begin{bmatrix}
-b \sin\theta\\
a \sin\theta\\
-a \cos\theta\\
0\\
\vdots\\
0
\end{bmatrix}.
\]
Clearly, $p_1^tv_\theta=p_2^tw_\theta=v_\theta^tw_\theta=0$. Define $f(\theta)=p_1w_\theta+p_2v_\theta= b\cos\theta - \dfrac{b\sin\theta}{\sqrt{b^2\sin^2\theta+a^2}}$ which is a continuous function with $f(0)=b$ and $f(\pi)=-b$. Hence there exists $\theta'\in[0,\pi]$ such that $f(\theta')=0$. Then we take $u_2=v_{\theta'}$ and $u_1=w_{\theta'}$.
\end{proof}

\begin{lemma}\label{include1}
For $P,Q\in\mathbb{R}^{n\times n}$, $n\geq 3$ and $0\leq \epsilon\leq 1$ we define
\[
P_\epsilon=\begin{bmatrix}\epsilon I_2 & \\ & I_{n-2}\end{bmatrix}P\quad \text{ and } \quad Q_\epsilon=\begin{bmatrix}\epsilon I_2 & \\ & I_{n-2}\end{bmatrix}Q.
\]
Then $(P_\epsilon,Q_\epsilon)\in \SA(P,Q)$.
\end{lemma}
\begin{proof}
For any $U\in \SO_n$, consider the ellipse $E(U)$ defined in Lemma~\ref{ellipse1}. If $\big(\text{tr}(P_\epsilon U),\text{tr}(Q_\epsilon U)\big)^t\in E(U)$, then we have $\big(\text{tr}(P_\epsilon U),\text{tr}(Q_\epsilon U)\big)^t\in\linear(P,Q;\SO_n)$. Now assume that $\big(\text{tr}(P_\epsilon U),\text{tr}(Q_\epsilon U)\big)^t\notin E(U)$. Then $\big(\text{tr}(P_\epsilon U),\text{tr}(Q_\epsilon U)\big)^t$ lies inside the ellipse $E(U)$. Since $\SO_n$ is path-connected, consider a continuous function $f:[0,1]\rightarrow \SO_n$ with $f(0)=U$ and $f(1)=U_0$ where $U_0$ is defined in Lemma~\ref{lemma6}. Since $E(f(1))$ degenerates, by continuity of $f$, there exist $s\in[0,1]$ such that  $\big(\text{tr}(P_\epsilon U),\text{tr}(Q_\epsilon U)\big)^t\in E(f(s))\subseteq \linear(P,Q;\SO_n)$. As it is true for all $U\in\SO_n$, we have $\linear(P_\epsilon ,Q_\epsilon ;\SO_n)\subseteq \linear(P,Q;\SO_n)$ and hence $(P_\epsilon ,Q_\epsilon )\in \SA(P,Q)$.
\end{proof}

Lemma~\ref{include1} remains valid if we consider $\SA_A(P,Q)$ instead of $\SA(P,Q)$. 

\begin{corollary}\label{include2}
Let $A\in\Rn$ and $n\geq 3$. For any $P,Q\in\mathbb{R}^{n\times n}$ and $0\leq \epsilon\leq 1$, we define
\[
P_\epsilon=\begin{bmatrix}\epsilon I_2 & \\ & I_{n-2}\end{bmatrix}P\quad \text{ and } \quad Q_\epsilon=\begin{bmatrix}\epsilon I_2 & \\ & I_{n-2}\end{bmatrix}Q.
\]
Then $(P_\epsilon,Q_\epsilon)\in \SA_A(P,Q)$.
\end{corollary}
\begin{proof}
For any $U,V\in\SO_n$, let $P'=PUAV$, $Q=QUAV$,
$
P'_\epsilon =(\epsilon I_2 \oplus I_{n-2})P'=P_\epsilon UAV$ and $Q'_\epsilon =(\epsilon I_2 \oplus I_{n-2})Q'=Q_\epsilon UAV .$
By Lemma~\ref{include1}, because $(P'_\epsilon,Q'_\epsilon)\in\SA(P',Q')$, there exists $W\in\SO_n$ such that 
\[
\begin{split}
\big(\tr(P_\epsilon UAV),\tr(Q_\epsilon UAV)\big)^t&=(\tr P'_\epsilon,\tr Q'_\epsilon)^t\\
& =\big(\tr(P'W),\tr (Q'W)\big)^t\\
& = \big(\tr(PUAVW),\tr(QUAVW)\big)^t\\&\in \linear(P,Q;O(A)).
\end{split}
\]
As this is true for all $U,V\in\SO_n$, we have $\linear(P_\epsilon,Q_\epsilon;O(A))\subseteq \linear(P,Q;O(A))$.
\end{proof}

Note that in Lemma~\ref{include1} and Corollary~\ref{include2}, $P_\epsilon,Q_\epsilon$ can be defined by picking arbitrary two rows of $P$ and $Q$ instead of the first two rows. We are now ready to prove our second main theorem.

\begin{proof}[Proof of Theorem~\ref{main_thm1}]
By Lemma~\ref{main_thm1_red}, it suffices to show that for all $P,Q\in\mathbb{R}^{n\times n}$, $\SA(P,Q)$ is star-shaped with respect to $(0_n,0_n)$. Let $(P',Q')\in \SA(P,Q)$ and $0\leq \alpha\leq 1$. We apply Lemma~\ref{include1} repeatedly to every two rows of $P,Q$. Then we have $(\epsilon^N P', \epsilon^N Q')\in \SA(P',Q')\subseteq \SA(P,Q)$ where $N=\frac{n!}{2(n-2)!}$. Taking $\epsilon=\sqrt[N]{\alpha}$, we have 
\[
\alpha (P',Q') = \alpha (P',Q') + (1-\alpha)(0_{n},0_{n})\in \SA(P,Q).
\] \end{proof}

For the case of $\ell=2$ and $\ell=3$, we know that $n=3$ and $n=4$ are respectively the smallest integers such that $L(O(A))$ is star-shaped for all $A\in\Rn$ and all linear maps $L:\Rn\to\R^\ell$. However, for $\ell\geq 4$, $n=2^{\ell-1}$ may not be the smallest integer to ensure star-shapedness of $L(O(A))$. One may ask the following question. 

\begin{problem}
For a given $\ell\geq 4$, what is the smallest $n$ such that $L(\SO_n)$ is star-shaped for all linear maps $L:\Rn\to\R^\ell$?
\end{problem}

The preceding results on star-shapedness of $L(O(A))$ can be easily generalized to the following joint orbits. We let $(\Rn)^m:=\{(A_1,...,A_m):A_1,...,A_m\in\Rn\}$. 

\begin{definition} For any $A_1,...,A_m\in\Rn$, we define
\[
\begin{split}
&\boldsymbol{O}_1(A_1,...,A_m;G):=\{(A_1V,...,A_mV):V\in G\},\\
&\boldsymbol{O}_2(A_1,...,A_m;G):=\{(UA_1,...,UA_m):U\in G\},\\
&\boldsymbol{O}_3(A_1,...,A_m;G):=\{(UA_1V,...,UA_mV):U,V\in G\},
\end{split}
\]
where $G=\On_n$ or $\SO_n$.
\end{definition}

\begin{theorem}\label{thm10}
Let $L:(\Rn)^m\rightarrow \mathbb{R}^\ell$ be linear, $(A_1,...,A_m)\in (\Rn)^m$ and $G=\On_n$ or $\SO_n$. If 
\begin{enumerate}
	\item[(i)] $\ell=2$ and $n\geq 3$, or
	\item[(ii)] $\ell \geq 3$ and $n\geq 2^{\ell-1}$,
\end{enumerate}
then $L(\boldsymbol{O}_i(A_1,...,A_m;G)),\; i=1,2,3$, are star-shaped with respect to the origin.
\end{theorem}
\begin{proof}
The case of $G=\On_n$ can be derived from the case $G=\SO_n$ easily. Hence we consider the case $G=\SO_n$ only and simply denote $\boldsymbol{O}_i(A_1,...,A_m;\SO_n)$ by $\boldsymbol{O}_i(A_1,...,A_m)$. For any given $L:(\Rn)^m\rightarrow \mathbb{R}^\ell$, express it by
\[
L(X_1,...,X_m)=\left(\text{tr}\left(\sum_{i=1}^m P^{(1)}_i X_i\right),...,\text{tr}\left(\sum_{i=1}^m P^{(\ell)}_i X_i\right)\right)^t,
\]
for some $P^{(j)}_i\in \mathbb{R}^{n\times n}$, $i=1,...,m,j=1,...,\ell$. 
For $\boldsymbol{O}_1(A_1,...,A_m$) we have 
\[
\begin{split}
&~L(\boldsymbol{O}_1(A_1,...,A_m))\\=&~\left\{\left( \text{tr}\left(\sum_{i=1}^m P^{(1)}_i A_iU\right),...,\text{tr}\left(\sum_{i=1}^m P^{(\ell)}_iA_iU\right)\right)^t:U\in \SO_n \right\} \\
 =&~\linear\left(\sum_{i=1}^m P^{(1)}_iA_i,...,\sum_{i=1}^m P^{(\ell)}_iA_i;\SO_n\right).
\end{split}
\]
Similarly for $L(\boldsymbol{O}_2(A_1,...,A_m))$. Hence the star-shapedness follows from Theorem~\ref{main_thm2} and Theorem~\ref{main_thm1}. 

Now consider the case of $\boldsymbol{O}_3(A_1,...,A_m)$. For any $U,V\in\SO_n$, we have
\[\begin{split}
L(UA_1V,...,UA_mV)&=\left(\text{tr}\left(\sum_{i=1}^m P^{(1)}_i UA_iV\right),...,\text{tr}\left(\sum_{i=1}^m P^{(\ell)}_i UA_iV\right)\right)^t\\
&\in \linear \left(\sum_{i=1}^m P^{(1)}_i UA_i,...,\sum_{i=1}^m P^{(\ell)}_i UA_i;\SO_N\right).
\end{split}\]
By star-shapedness of $\linear\left(\sum_{i=1}^m P^{(1)}_i UA_i,...,\sum_{i=1}^m P^{(\ell)}_i UA_i;\SO_N\right)$, for any $0\leq \alpha\leq 1$ we have
\[\begin{split}
\alpha L(UA_1V,...,UA_mV)& \in \linear\left(\sum_{i=1}^m P^{(1)}_i UA_i,...,\sum_{i=1}^m P^{(1)}_i UA_i;\SO_N\right)^t\\ &\subseteq L(\boldsymbol{O}_3(A_1,...,A_m)).\end{split}
\] 
\end{proof}

\section{Convexity of linear image of $O(A)$}
We first give two non-convex examples, one is a linear image of $O(A)$ under $L:\Rn\to\R^\ell$ with $\ell\geq 3$ and another is a linear image of $\boldsymbol{O}_3(A_1,...,A_m)$ under $L:(\Rn)^m\to\R^\ell$ with $\ell\geq 2$.


\begin{example}
Consider $O(I_n)=\SO_n$ with $n\geq 2$ and the linear map $L:\Rn\to\R^\ell$ with $\ell\geq 3$ defined by
\[
L(X)=(\tr(P_1X),...,\tr(P_\ell X))^t
\]
where 
\[
P_1=I_{n-2}\oplus 0_2,\;\;\;P_2=I_{n-2}\oplus \begin{bmatrix}1&0\\0&0\end{bmatrix},\;\;\; P_3=I_{n-2}\oplus\begin{bmatrix}0&1\\0&0\end{bmatrix},
\]
and $P_j=0_n$ for $j=4,...,\ell$. The mid-point of points
$L(I_n)=(n-2,n-1,n-2,0,...,0)^t$ and $L\left(I_{n-2}\oplus \begin{bmatrix}0&-1\\1&0\end{bmatrix}\right)=(n-2,n-2,n-1,0,...,0)^t$ is in $L(P_1,...,P_\ell;\SO_n)$ only if there exists $U\in\SO_n$ having the form 
\[
U=I_{n-2}\oplus \begin{bmatrix}u_{11} & u_{12}\\u_{21} &u_{22} \end{bmatrix} 
\] with $u_{11}=\frac{1}{2}=u_{21}$. This is impossible as $u_{11}^2+u_{21}^2=1$. Hence $L(\SO_n)$ is non-convex.
\end{example}

\begin{example}
For $n\geq 3,\; m\geq 2,\; \ell\geq 2$, consider the matrices,
\[
A_1=\begin{bmatrix}1& 0& 0\\0 &0 & 0\\0&0&0\end{bmatrix}\oplus 0_{n-3}, \;\;A_2=\begin{bmatrix}0& 0& 0\\0 &1 & 0\\0&0&0\end{bmatrix}\oplus 0_{n-3},\;\;A_j=0_n,\;\;j=3,...,m,
\]
and the linear map $L:(\Rn)^m\to\R^\ell$ defined by
\[
L(X_1,...,X_m):=\big(\tr(A_1X_1+A_2X_2),\tr(A_2X_1-A_1X_2),0,...,0\big)^t.
\]
By taking $U=V=I_n$, and $U=\begin{bmatrix} 0& 1&0\\  1& 0&0 \\ 0& 0& -1\end{bmatrix}\oplus I_{n-3},$ $V=\begin{bmatrix} 0&1 &0\\  -1&0 &0 \\0 & 0& 1\end{bmatrix}\oplus I_{n-3}$ respectively, we have $(2,0,0,...,0)^t,(0,2,0,...,0)^t\in L(\boldsymbol{O}_3(A_1,...,A_m))$. We shall show that their mid-point which is $(1,1,0,...,0)^t\notin L(\boldsymbol{O}_3(A_1,...,A_m))$. For any $U=[u_{ij}],\; V=[v_{ij}]\in\SO_n$, by direct computation we have
\[
UA_1V=\begin{bmatrix}u_{11}v_{11}& \ast &\ast  \\ \ast & u_{21}v_{12}&\ast \\ \ast&\ast&\ast\end{bmatrix},\;\;\;\;UA_2V=\begin{bmatrix}u_{12}v_{13}& \ast &\ast  \\ \ast & u_{22}v_{22}&\ast \\ \ast&\ast&\ast\end{bmatrix}.
\]
Hence $(1,1,0, ...,0)\in L(\boldsymbol{O}_3(A_1,...,A_m))$ only if $u_{11}v_{11}+u_{22}v_{22}=1=u_{21}v_{12}-u_{12}v_{13}$ for some $U,V\in\SO_n$. We shall show that such $U,V$ do not exist. For $X=(x_{ij}),\; Y=(y_{ij})\in\Rn$, denote $X\circ Y:=(x_{ij}y_{ij})\in \Rn$. Since each absolute row (column) sum of $U\circ V$ is not greater than one, we have $(1,1,0,...,0)\in L(\boldsymbol{O}_3(A_1,...,A_m))$ only if there exist $U,V\in\SO_n$ such that
\[
U\circ V=\begin{bmatrix}\dfrac{1}{2} & \dfrac{1}{2} & 0 \\ \\  -\dfrac{1}{2} & \dfrac{1}{2}&0\\ \\0&0&\ast \end{bmatrix} \;\;\; \text{or} \;\;\; U\circ V=\begin{bmatrix}\dfrac{1}{2} & -\dfrac{1}{2} &0 \\ \\  \dfrac{1}{2} & \dfrac{1}{2}&0\\ \\0&0&\ast \end{bmatrix} .
\]
The possible choices of the leading $2\times 2$ principal submatices of $U$ and $V$ are
\[
\pm\frac{\sqrt{2}}{2}\begin{bmatrix}1& k_1\\ -k_2&k_1k_2\end{bmatrix} 
\]
where $k_1,k_2=\pm 1$. However, any two of them will not give the $U\circ V$ as required.
\end{example}

From the above two examples we know that $L(O(A))$ is not convex in general. However if the codomain of $L$ is $\R^2$ then $L(O(A))$ is always convex. This result was obtained by Li and Tam \cite{LT} by using techniques in Lie algebra. In the following, we shall give an alternative proof on this result by showing that $L(O(A))$ has convex boundary for all $A\in \Rn$ and linear $L:\Rn\to\R^2$, i.e., the intersection of $L(O(A))$ with any of its supporting lines is path connected. Combining with the star-shapedness property of $L(O(A))$, the convexity of $L(O(A))$ follows. We first need some notations.

\begin{definition}
For $A=(a_{ij})\in \mathbb{R}^{n\times n}$, we denote its diagonal as $d(A)=(a_{11},a_{22},...,a_{nn})^t\in\R^n$. We further denote the sum of the first $k$ diagonal elements of $A$ by $t_k(A)$. Moreover for $P\in\Rn$, we denote $r(P,A)= \max\{\tr(PUAV):{U,V\in \SO_n}\}$ and $\GA_P(A)=\{B\in O(A):\tr(PB)=r(P,A) \}$. 
\end{definition}

We shall characterize the set $\GA_P(A)$ when $A$ has distinct singular values and then show that it is path connected. Note that for any $U,V\in\SO_n$, $\GA_P(UAV)=\GA_P(A)$ and $\GA_{UPV}(A)=\{V^tBU^t:B\in \GA_P(A)\}$. Hence we may assume that $A,P$ are diagonal matrices.

\begin{lemma}\label{lem20}
Let $A=\dg(a_1,...,a_{n-1},a_n)$ where $a_1>a_2>\cdots>a_{n-1}>|a_n|\geq 0$ and $B\in O(A)$. If $t_k (B)=t_k (A)$ then 
\[
B=\begin{bmatrix}W & \\ & X_1\end{bmatrix}A\begin{bmatrix}W^t & \\ & X_2\end{bmatrix},
\]
where $W\in \SO_k$, $X_1,X_2\in \SO_{n-k}$.
\end{lemma}
\begin{proof}
Let $B=UAV$ where $U,V\in\SO_n$ and write 
\[
U=(u_{ij})=\begin{bmatrix} U_{11} & U_{12}\\U_{21} & U_{22} \end{bmatrix}, \;\;\; V=(v_{ij})=\begin{bmatrix} V_{11} & V_{12} \\V_{21} &V_{22} \end{bmatrix}, 
\]
where $U_{11},V_{11}\in\mathbb{R}^{k\times k}$, $U_{22},V_{22}\in\mathbb{R}^{(n-k)\times (n-k)}$. Denote
\[
\begin{bmatrix} U_{11} & U_{12}\end{bmatrix}=\begin{bmatrix}u_{\ast 1} &\cdots &u_{\ast n} \end{bmatrix}, \;\;\; \begin{bmatrix} V_{11} \\ V_{21}\end{bmatrix}= \begin{bmatrix}v_{1\ast}\\ \vdots\\ v_{n \ast}\end{bmatrix},
\]
where $u_{\ast j}^t=(u_{1j},...,u_{kj}),v_{j\ast}=(v_{j1},...,v_{jk})$, $j=1,...,n$. Then $t_k (UAV)=\text{tr} (U_{11} A_{11} V_{11} + U_{12}A_{22}V_{21})=\text{tr} ( A_{11} V_{11}U_{11} + A_{22}V_{21}U_{12})=\sum_{i=1}^n a_{i}v_{i\ast}u_{\ast i}$. Since $v_{i\ast }u_{\ast i}\leq 1$, $\sum_{i=1}^n v_{i\ast}u_{\ast i}\leq k$ and $a_{1}>\cdots >a_{k}>\cdots>a_{n}$, we have $\sum_{i=1}^n a_{i}v_{i\ast}u_{\ast i}\leq \sum_{i=1}^k a_{ii}$ with equality holds if and only if $v_{i\ast}u_{\ast i}=1$ for $i\leq k$ and $v_{i\ast}u_{\ast i}=0$ for $i>k$. Hence we have $v_{i\ast}=u_{\ast i}^t$ and $u_{\ast i}u_{\ast i}^t=1$. Now $U=W\oplus X_1$ and $V=W^t\oplus X_1$ where $W\in\On_k$, $X_1,X_2\in\On_{n-k}$ and $\deter W= \deter X_1=\deter X_2.$ If $\deter W= \deter X_1=\deter X_2=-1$, then we have $B=\left((WD_1)\oplus (X_1D_2)\right)A\left((WD_1)^t\oplus (D_2X_2)\right)$ where $D_1=I_{k-1}\oplus -1$ and $D_2=-1 \oplus I_{n-k-1}$.
\end{proof}

Thompson \cite{Thompson77} gave the following result on characterizing the diagonal elements of $O(A)$.

\begin{proposition}\cite{Thompson77}\label{Thompson} A vector $d=(d_1,..., d_n)$ is the diagonal of a matrix $A\in\Rn$ with singular values $s_1\geq s_1\geq\cdots\geq s_n$ if and only if $d$ lies in the convex hull of those vectors $(\pm s_{\sigma(1)}, ..., \pm s_{\sigma(n)})$ with an even number (possibly zero) of negative signs and arbitrary permutation $\sigma$.
\end{proposition}

For matrices $A,B\in\Rn$, the following result by Miranda and Thompson \cite{Miranda} can be regarded as a characterization of the extreme values of $O(A)$ under the linear map $X\longmapsto \tr(BX)$.

\begin{proposition}\cite{Miranda}\label{miranda1}
Let $A,B\in\Rn$ have singular values $s_1(A)\geq \cdots\geq s_n(A)$ and $s_1(B)\geq \cdots\geq s_n(B)$ respectively. Then
\[
\max_{U,V\in \SO_n} \tr (BUAV) =\sum_{i=1}^{n-1}s_i(A)s_i(B)+(\mathrm{ sign~det}(AB))s_n(A)s_n(B).
\]
\end{proposition}

\begin{theorem}\label{thm22}
Let $A=\dg(a_1,...,a_{n-1},\pm a_n)$ where $a_1>\cdots>a_n\geq 0$ and $P=p_1I_{n_1}\oplus\cdots\oplus p_kI_{n_k}$ where $p_1>\cdots>p_k\geq 0$ and $n_1+\cdots+n_k=n$. Then
\begin{enumerate}
\item[(i)] if $p_k>0$,
\[
\GA_P(A)=\left\{\begin{bmatrix} U_1 & & \\ & \ddots &\\ & & U_k \end{bmatrix}A \begin{bmatrix} U^t_1 & & \\ & \ddots &\\ & & U^t_k\end{bmatrix}: \begin{aligned} &U_i\in \SO_{n_i},\\ &i=1,...,k \end{aligned}\right\};
\] 
\item[(ii)] if $p_k=0$,
\[
\GA_P(A)=\left\{\arraycolsep=1.4pt\def\arraystretch{1}\left[\begin{array}{llll} U_1 & & &\\ & \ddots & &\\ & & U_{k-1} &\\ & & & U\end{array}\right]A \left[\begin{array}{cccc} U^t_1 & & & \\ & \ddots & &\\ & & U^t_{k-1}& \\ & & & V\end{array}\right]: 
\begin{aligned}
&U_i\in \SO_{n_i},\\&i=1,...,k-1,\\ &U,V\in \SO_{n_k}
\end{aligned}\right\}.
\]
\end{enumerate}
In both cases, $\GA_P(A)$ is path connected.
\end{theorem}

\begin{proof} $(\supseteq)$ Obvious. $(\subseteq)$. We assume that $A=A_1\oplus\cdots\oplus A_k$ where $A_i\in\mathbb{R}^{n_i\times n_i}$. We have $r(P,A)=d(P)^t d(A)=\sum_{i=1}^k p_i \tr A_i$. Let $U,V\in\SO_n$ such that $\tr(PUAV)=r(P,A)=d(P)^td(UAV)$. Write 
\[
UAV=B=\begin{bmatrix}
B_{11} & B_{12} & \cdots & B_{1k}\\
B_{21} & B_{22} & \cdots & B_{2k}\\
\vdots & \cdots & \ddots & \vdots\\
B_{k1} & B_{k2} & \cdots & B_{kk}
\end{bmatrix}
\]
where $B_{ij}\in\mathbb{R}^{n_i\times n_j}$. We have $\tr(PUAV)=\tr(PB)=\sum_{i=1}^k p_i\tr B_{ii}$. We shall show that $\tr B_{ii}=\tr A_i$ for all $i$ whenever $p_i>0$. By Proposition~\ref{Thompson}, $\Dg(B)=\sum \alpha_i s_i$ where $\alpha_i> 0$, $\sum \alpha_i=1$ and $s_i$ are vector of $(\pm a_{\sigma(1)},...,\pm a_{\sigma(n)})$, $\sigma$ is a permutation on $\{1,...,n\}$ and the number of negative signs is even (odd, respectively) if $\deter A\geq 0$ ($\leq 0$, respectively). If $k=1$, then $P=p_1 I$, and the proof is trivial. Now consider $k>1$, hence $p_1>0$. We first show that $\tr B_{11}=\tr A_1$. Note that $\tr B_{11}<\tr A_1$ holds if and only if at least one of the following cases hold: 
\begin{enumerate}
\item[(1)] there exists $i_1$ such that the first $n_1$ elements of $s_{i_1}$ contain $-a_j$ where $j\leq n_1$;
\item[(2)] there exists $i_1$ such that the first $n_1$ elements of $s_{i_1}$ contain $\pm a_j$ where $j>n_1$.
\end{enumerate}
In case (1), we construct $s'_{i_1}$ from $s$ by multiplying $-1$ to $-a_j$ and arbitrary $a_q$ for some $q>n_1$. If in case (2), then there exists $i'<n_1$ such that $\pm a_{i'}$ will not be the first $n_1$ elements of $s_{i_1}$. In this case, we construct $s'_{i_1}$ from $s_{i_1}$ by interchanging $\pm a_j$ and $\pm a_{i'}$ and multiplying $-1$ to both if necessary to have $a_{i'}$ instead of $-a_{i'}$. Replace $s_{i_1}$ in $\sum \alpha_i s_i$ by $s'_{i_1}$ to form $s$. By Proposition~\ref{Thompson}, there exists $B'\in O(A)$ such that $\Dg(B')=s$. We shall have $\Dg(P)^t\Dg(B)= \Dg(P)^t(\sum \alpha_i s_i) = \Dg(P)^ts + \Dg(P)^t(s_{i_1}-s'_{i_1})<\Dg(P)^ts$, which contradicts the assumption on $B$. Therefore, we have $\tr B_{11}=\tr A_1$. By Lemma~\ref{lem20}, we have $U=U_1\oplus U_2$ and $V=V_1^t\oplus V_2$ where $U_1,V_1\in\SO_{n_1}$, $V_2,U_2\in \SO_{n-n_1}$ and $V^t_1=U_1$. Apply similar approach for $B_{ii}$ where $p_i>0$. Hence, if $p_k>0$, we have $U=U_1\oplus\cdots\oplus U_k$ and $V=U^t$ where $U_i\in \SO_{n_i}$, $i=1,...,k$; otherwise if $p_k=0$, $U=U_1\oplus\cdots\oplus U_{k-1}\oplus U'$ and $V=U^t_1\oplus\cdots\oplus U^t_{k-1}\oplus V'$ where $U_i\in \SO_{n_i}$, $i=1,...,k-1$, $U',V'\in \SO_{n_k}$. The path connectedness of $\GA_P(A)$ follows from the path connectedness of $\SO_{n_i}$ for all $i$.
\end{proof}

\begin{corollary}\label{coro22}
If $A \in\mathbb{R}^{n\times n}$ has $n$ distinct singular values, then $L(O(A))$ has convex boundary for all linear maps $L:\mathbb{R}^{n\times n}\rightarrow \mathbb{R}^2$.
\end{corollary}
\begin{proof}
Let $P,Q\in\mathbb{R}^{n\times n}$ be such that $\linear(P,Q;O(A))=L(O(A)).$ Then $L(O(A))$ has convex boundary if for any $\theta\in [0,2\pi]$, the set
\[
\{-\sin\theta x+ \cos\theta y:(x,y)\in \linear(P,Q;O(A)),~\cos\theta x+\sin\theta y =r_\theta\},
\]
where $r_\theta=\max \{\cos\theta x+\sin\theta y:(x,y)\in \linear(P,Q;O(A))\}$, is path connected. For any $\theta\in [0,2\pi]$, we define $P'_\theta=-\sin\theta P+ \cos\theta Q$ and $Q'_\theta=\cos\theta P+\sin\theta Q$, then we have
\[
\begin{split}
 & ~\{-\sin\theta x+ \cos\theta y:(x,y)\in \linear(P,Q;O(A)),~\cos\theta x+\sin\theta y =r_\theta\}\\
 =&~\{\tr\left(P'_\theta UAV\right) :U,V\in \SO_n,\tr\left(Q'_\theta UAV\right) =r_\theta\}\\
  =&~\{\tr(P'_\theta X) :X\in \GA_{Q'_\theta}(A)\}
\end{split}
\]
Hence by Theorem~\ref{thm22}, it is path connected. 
\end{proof}

Note that a set $M\subseteq\R^2$ is convex if and only if it is star-shaped and has convex boundary. Hence by Theorem~\ref{thm10} and Corollary~\ref{coro22}, the following result is clear.

\begin{theorem}\label{thm_distinct_sv}
Let $n\geq 3$. If $A \in\mathbb{R}^{n\times n}$ has $n$ distinct singular values, then $L(O(A))$ is convex for all linear maps $L:\mathbb{R}^{n\times n}\rightarrow \mathbb{R}^2$.
\end{theorem}

In fact, the condition of distinct singular values in Theorem~\ref{thm_distinct_sv} can be removed by applying the following lemma.

\begin{lemma}\label{lem24}
Let $L:\mathbb{R}^{n\times n}\rightarrow \mathbb{R}^\ell$ be a linear map. Suppose $L(O(A))$ is convex for all $A$ in a dense set $S$ of $\mathbb{R}^{n\times n}$. Then $L(O(A))$ is convex for all $A\in\mathbb{R}^{n\times n}$.
\end{lemma}
\begin{proof}
Suppose that $A_0\in\mathbb{R}^{n\times n}$ such that $L(O(A_0))$ is not convex. Then there exist $x_1,x_2\in L(O(A_0))$ such that $y=\frac{1}{2}(x_1+x_2) \notin L(O((A_0))$. Since $L(O(A_0))$ is compact, there exists $\epsilon>0$ such that $B(y,\epsilon):=\{x\in \mathbb{R}^\ell: \left\|x-y\right\|<\epsilon\}$ has empty intersection with $L(O(A_0))$. Since $S$ is dense in $\mathbb{R}^{n\times n}$, there exists $A_\epsilon\in S$ such that for all $U,V\in \SO_n$,
\[
\left\| L(UA_0V)-L(UA_\epsilon V)\right\| < \frac{\epsilon}{2}.
\]
Hence there exist $x'_1,x'_2 \in L(O(A_\epsilon))$ such that $\left\|x'_1-x_1\right\|<\frac{\epsilon}{2}$ and $\left\|x'_2-x_2\right\|<\frac{\epsilon}{2}$. By convexity of $L(O(A_\epsilon))$, $y'=\frac{1}{2}(x'_1+x'_2)\in L(O(A_\epsilon)$. We have
 \[
\left\|y'-y\right\|=\left\|\frac{1}{2}(x'_1+x'_2)-\frac{1}{2}(x_1+x_2)\right\|<\frac{1}{2}\left(\frac{\epsilon}{2}+\frac{\epsilon}{2}\right)=\frac{\epsilon}{2}.
\] By assumption of $A_\epsilon$, there exists $z\in L(O(A_0))$ such that $\left\|z-y'\right\|<\frac{\epsilon}{2}$. Then $\left\|z-y\right\|=\left\|(z-y')+(y'-y)\right\|<\left\|(z-y')\right\|+\left\|(y'-y)\right\|<\frac{\epsilon}{2}+\frac{\epsilon}{2}=\epsilon$, contradicting the fact that $B(y,\epsilon) \cap L(O(A_0))=\emptyset$. 
\end{proof}

Since the set of $n\times n$ matrices with $n$ distinct singular values is dense in $\Rn$, by Lemma~\ref{lem24} we have the following result.

\begin{theorem}
Let $n\geq 3$. $L(O(A))$ is convex for all linear maps $L:\mathbb{R}^{n\times n}\rightarrow \mathbb{R}^2$ and $A\in\mathbb{R}^{n\times n}$.
\end{theorem}

From the proof of Corollary~\ref{thm10}, the convexity of $L(O(A))$ can be extended to $L(\boldsymbol{O}_i(A_1,...,A_m))$, $i=1,2$.

\begin{corollary}
Let $n\geq 3$. $L(\boldsymbol{O}_i(A_1,...,A_m))$, $i=1,2$, is convex for all linear maps $L:(\R^{n\times n})^m\rightarrow \mathbb{R}^2$ and $A_1,...,A_m\in\Rn$.
\end{corollary}


\begin{thebibliography}{9}
\addcontentsline{toc}{chapter}{Bibliography}
\bibitem{CheungT} W.S. Cheung, N.K. Tsing, {The $C$-numerical Range of Matrices is Star-shaped}, Linear and Multilinear Algebra {41} (1996), 245-250.
\bibitem{Hausdorff} F. Hausdorff, {Das Wertvorrat einer Bilinearform}, Math. Zeit. {3} (1919), 314-316.
\bibitem{HJ} A. Horn, C.R. Johnson, {Topics in Matrix Analysis}, Cambridge University Press, Cambridge, 1991.
\bibitem{ckli} C.K. Li,  $C$-numerical Ranges and $C$-numerical Radii, Linear and Multilinear Algebra, 37 (1994), 51-82.
\bibitem{LP1999} C.K. Li, Y.T. Poon, {Convexity of the Joint Numerical Range}, SIAM J. Matrix Analysis Appl. 21 (1999), 668-678.
\bibitem{LP2011} C.K. Li, Y.T. Poon, {Generalized Numerical Ranges and Quantum Error Correction}, J. Operator Theory. 66 (2011), 335-351.
\bibitem{LT} C.K. Li, T.Y. Tam, {Numerical Ranges Arising from Simple Lie Algebras}, Canad. J. Math. {52} (2000), 141-171.
\bibitem{Miranda} H. Miranda , R.C. Thompson, {Group Majorization, the Convex Hulls of Sets of Matrices, and the Diagonal Element-Singular Value Inequalities}, Linear Algebra Appl. {199} (1994), 131-141.
\bibitem{Thompson77} R.C. Thompson, {Singular Values, Diagonal Elements, and Convexity}, SIAM J. Appl. Math. {32} (1977), 39-63.
\bibitem{Toeplitz} O. Toeplitz, {Das algebraische Analogon zu einem Satze von Fejer}, Math. Zeit. {2} (1918), 187-197.
\bibitem{Tsing81} N.K. Tsing, {On the Shape of the Generalized Numerical Ranges,} Linear and Multilinear Algebra, 10 (1981), 173-182.
\bibitem{Westwick} R. Westwick, {A Theorem on Numerical Range}, Linear and Multilinear Algebra, {2} (1975), 311-315
\end{thebibliography}
\end{document}